\documentclass{amsart}

\usepackage{amsmath,amsthm,amsfonts,amssymb,amscd}
\usepackage[all]{xy}
\usepackage[english]{babel}

\usepackage[dvips]{graphicx}
\usepackage[T1]{fontenc}
\usepackage[latin1]{inputenc}
\usepackage{ae}

{\theoremstyle{plain}
\newtheorem{theorem}{Theorem}[section]
\newtheorem{proposition}{Proposition}[section]
\newtheorem{lemma}{Lemma}[section]

\newtheorem{remark}{Remark}[section]
\newtheorem{corollary}{Corollary}[section]
}
\newcommand{\nc}{\newcommand}
\nc {\hh}{\check{h}}
\nc {\DD}{\mathcal{D}}
\nc {\CC}{\mathbb{C}}
\nc {\Pp}{\mathbb{P}}
\nc {\Ss}{\mathcal{S}}
\nc {\PP}{\mathbb{P}^{n}}
\nc {\Pd}{ \check{\mathbb{P}}^{n}}
\nc {\WW}{\mathcal{W}}
\nc {\Sym}{\mathrm{Sym}}
\nc {\OO}{\mathcal{O}}
\nc {\UU}{\mathcal{U}}
\nc {\EE}{\mathcal{E}}
\nc {\MM}{\mathcal{M}}
\nc {\KK}{\mathcal{K}}
\nc {\PW}{\mathcal{P}}
\nc {\NW}{\mathcal{N}_{\WW}}
\nc {\FF}{\mathcal{F}}
\nc {\GG}{\mathcal{G}}
\nc {\ZZ}{\mathcal{Z}}
\nc {\LL}{\mathcal{L}}
\nc {\HH}{\mathcal{H}}
\nc {\NN}{\mathcal{N}}
\nc {\VV}{\mathcal{V}}
\nc {\Ww}{\mathbb{W}}
\nc {\QQ}{\mathbb{Q}}
\nc {\II}{\mathcal{I}}

\date{}

\begin{document}

\title[Characteristic Numbers and Invariant Subvarieties]{Characteristic Numbers and Invariant Subvarieties for Projective Webs}

\author{M. Falla Luza$^1$}
\author{T. Fassarella$^2$}

\thanks{1. Departamento de An\'alise -- IM -- UFF  \\
              M\'ario Santos Braga s/n -- Niter\'oi, 24.020-140 RJ Brasil\\ 
             maycolfl@impa.br}

\thanks{2. Departamento de An\'alise -- IM -- UFF  \\
              M\'ario Santos Braga s/n -- Niter\'oi, 24.020-140 RJ Brasil\\ 
             tfassarella@id.uff.br}
             
\begin{abstract}
We define the characteristic numbers of a holomorphic $k$-- distribution of any dimension on $\PP$ and obtain relations between these numbers and the characteristic numbers of an invariant subvariety. As an application we bound the degree of a smooth invariant hypersurface.
\end{abstract}
\maketitle 


\section*{Introduction}

The aim of this work is to relate the characteristic numbers of projective $k$--webs, or more generally, $k$--distributions of arbitrary dimension to those of invariant subvarieties. Loosely speaking, a $k$--distribution $\WW$ of dimension $p$ on $\PP$ is locally given by $k$ holomorphic fields of $p$--planes on the complement of a Zariski closed set. The most basic invariants attached to it are its characteristic numbers $d_0,...,d_p$ where $d_i$ is defined as the degree of the tancency locus of the distribution with a generic $\mathbb P^{n-p+i-1}$ linearly embedded in $\PP$. Suppose now $V \hookrightarrow \PP$ is a subvariety invariant by $\WW$. Our goal is to obtain inequalities involving the characteristic numbers of $V$ and $\WW$. As a corollary we give some bounds for the degree of a smooth invariant hypersurface. 

The question of bounding the degree of an algebraic curve which is a solution of a foliation on $\mathbb P^2$ in terms of the degree of the foliation was treated by H. Poincar\'e in \cite{Po}. Versions of this problem have been considered in a number of recent works, see for example \cite{Soa} and references therein. 
In that paper, M. Soares considered one dimensional projective foliations and their tangency locus with a pencil of hyperplanes. This tangency locus is an analogous for foliations of the polar classes for projective varieties. For a variety $V$ of dimension $q$ invariant by a one dimensional foliation $\mathcal{F}$ he compaired their polar classes and obtained the relation $\deg(P^{V}_{q-j})\le \deg(P^{V}_{q-j-1})\cdot(\deg(\mathcal F) + 1)$, where $P^{V}_{k}$ is the $k$--th polar class of $V$ and $j$ is some number between $0$ and $q-1$, see \cite[Theorem 1]{Soa}. As a consequence the bound $\deg(V)\le\deg(\mathcal F)+2$ was obtained for a smooth invariant hypersurface. Polar classes was also considered by R. Mol in \cite{Mol} for holomorphic distributions of arbitrary dimension. He expresses these classes in terms of the Chern-Mather classes of the tangent sheaf of the distribution, moreover, Theorem 1 of \cite{Soa} is generalizated. 

In this paper we associate to any $k$--distribution a subvariety of $\mathbb P(T^* \PP)$. When we write its cohomology class, the characteristic numbers appear naturally. We also define, as in the case of foliations, the polar classes of a $k$--distribution and get theirs degrees in terms of the characteristic numbers. Then we consider a subvariety which is invariant by a $k$--distribution and relate the polar classes of them obtaining more relations than the known for distributions, see Theorem \ref{T:theorem}. As a consequence we obtain as many bounds for the degree of a smooth invariant hypersurface as the dimension of the $k$--distribution, see Corollary \ref{C:corollary}.

\section{Characteristic numbers of projective webs}

Let $\PP$ be the $n$-dimensional complex projective space and $M=\Pp(T^{*}\PP)$ the projectivization of its cotangent bundle. Since $M$ can be identify with the incidence variety of points and hyperplanes in $\PP$, one has two natural projections
$$
\xymatrix{
&M=\Pp(T^{*}\PP) \ar[dl]_{\pi} \ar[dr]^{\check{\pi}} &\subseteq \PP \times \Pd \\
\PP& &\Pd} 
$$

Let us denote by $h=c_{1}(\OO_{\PP}(1))$ and $\hh=c_{1}(\OO_{\Pd}(1))$ the hyperplane classes on $\PP$ and $\Pd$ respectively. We still denote by $h$ and $\hh$ the respective pullbacks to $M$ by $\pi$ and $\check{\pi}$. Note that the cohomology ring $H^{*}(M)$ is, via the pullback map $\pi^{*}:H^{*}(\PP)\rightarrow H^{*}(M)$, an algebra over the ring $H^{*}(\PP)$, wich is generated by $\xi = c_{1}(\mathcal{O}_{M}(-1))$, the Chern class of the tautological bundle $\mathcal{O}_{M}(-1)$, with the relation 
$
\sum_{i=0}^{n}\binom{n+1}{i+1}h^{n-i}\xi^{i}=0
$ 
(see \cite{GH}, pg. 606).

Observe that $h^{n}$ is the class of a fiber of $\pi$ and the restriction of $\mathcal{O}_{M}(-1)$ to each fiber is the universal bundle, so that
 $\int_M\xi^{n-1}h^{n}=(-1)^{n-1}$ and $\int_M\xi^{n}h^{n-1}=(-1)^{n}(n+1)$, where the last equation follows from the previous relation. Then if we write $\hh = a h + b \xi$ it is easy to see that $b=-1$ and therefore we get the following description of $H^{*}(M)$

\begin{equation*} \label{cohomologia-de-M}
H^{*}(M)=\frac{\mathbb{Z}[h,\hh]}{\langle h^{n+1}, h^{n} - h^{n-1} \hh + \ldots + (-1)^{n} \hh^{n} \rangle}.
\end{equation*}
Clearly we also have the relations $\hh^{n+1}=0$, $\int_Mh^{n}\hh^{n-1}=\int_Mh^{n-1}\hh^{n}=1$.

Let $V \subseteq \PP$ be an irreducible projective subvariety, the \textbf{conormal variety} of $V$ is defined as ${\rm{Con}} (V)= \overline{\Pp(N^{*}V_{sm})}$, where $V_{sm}$ denotes the smooth part of $V$ and $N^{*}V_{sm}$ its conormal bundle. We notice that via the identification $M\subset \PP \times \Pd$, ${\rm{Con}}(V)$ is the closure of the set of pairs $(x,H)$ such that $x$ is a smooth point of $V$ and $H$ is a hyperplane containing the tangent plane $T_xV$.  For example, the conormal variety of a point $\Pp^{0} \subseteq \PP$ is all the fiber $\pi ^{-1}(\Pp^{0})$, so its class is $h^{n}$. More generally one has the following lemma.

\begin{lemma}\label{conormal-de-Pi}
The conormal variety of a linearly embedded $\Pp^{j} \subseteq \PP$ is a trivial $\Pp^{n-j-1}$ bundle over $\Pp^{j}$ which class is given by $$[{\rm{Con}}(\Pp^{j})]=(-1)^{j}h^{n}+ \ldots +h^{n-j+2}\hh^{j-2}-h^{n-j+1}\hh^{j-1}+h^{n-j}\hh ^{j}.$$
\end{lemma}

\begin{proof}
Recall that ${\rm{Con}}(\Pp^{j})=\{(p,H)\in M : p \in \Pp^{j}, H \supseteq \Pp^{j}\}$ is an irreducible subvariety of $M$ of codimension $n$, so we can write $$[{\rm{Con}}(\Pp^{j})]=a_{n}h^{n} + a_{n-1}h^{n-1}\hh + \ldots + a_{1}h\hh^{n-1}$$ and use the before relations to get 
$$1=\int_M[{\rm{Con}}(\Pp^{j})]\cdot h^{j}\cdot \hh^{n-j-1}=a_{n-j}+a_{n-j-1}$$ 
and 
$$0=\int_M[{\rm{Con}}(\Pp^{j})]\cdot h^{k}\cdot\hh^{n-k-1}=a_{n-k}+a_{n-k-1}$$
for $k \in \{0,1,\ldots, j-1,j+1, \ldots n-1\}$ (here $a_{0}=0$).  The lemma follows from the previous equalities. 
\end{proof}
For any projective subvariety $V \subseteq \PP$ of dimension $q$ we define its \textbf{characteristic numbers} as the integers $a_{i}'s$ such that $$[{\rm{Con}}(V)]= a_{n}h^{n}+a_{n-1}h^{n-1}\hh + \ldots +a_{1}h \hh^{n-1}.$$ For convenience we fix $a_0=0$ and in particular we have 
$$
\deg(V)=\int_M[{\rm{Con}}(V)]\cdot h^{q}\cdot \hh^{n-q-1}= a_{n-q} + a_{n-q-1}. 
$$

Now we refer to \cite[Section 1.3]{PiPer} for more details on the following definitions. Fix $k, p\in \mathbb N$ with $1\le p <n$. Roughly speaking, to give a $k$--distribution of dimension $p$ is the same to give, over a generic point, a set of $k$ various $p$-dimensional planes, varying holomorphically. More precisely, a \textbf{$k$--distribution $\WW$ of dimension $p$ on $\PP$} is given by an open covering $\mathcal{U}=\{U_{i}\}$ of $\PP$ and $k$-symmetric $(n-p)$-forms $\omega_{i} \in \Sym^{k}\Omega^{n-p}_{\PP}(U_{i})$ subject to the conditions:
\begin{enumerate}
 \item For each non-empty intersection $U_{i} \cap U_{j}$ there exists a non-vanishing function $g_{ij} \in \OO_{U_{i} \cap U_{j}}$ such that $\omega_{i}= g_{ij} \omega_{j}$.
 \item For every $i$ the zero set of $\omega_{i}$ has codimension at least two.
 \item For every $i$ and a generic $x\in U_i$, the germ of $\omega_{i}$ at $x$ seen as homogeneous polynomial of degree $k$ in the ring $\mathcal O_x [...,dx_{i_1}\wedge\dots\wedge dx_{i_{n-p}},...]$ is square--free.
 \item For every $i$ and a generic $x \in U_{i}$, the germ of $\omega_{i}$ at $x$ is a product of $k$ vaious $(n-p)$-forms $\beta_{1}, \ldots,  \beta_{k}$, where each $\beta_{i}$ is a wedge product of $(n-p)$ linear forms. 
\end{enumerate}
If in addition the forms $\beta_{i}$ are integrable we will say that the distribution is a \textbf{$k$-web of dimension $p$ on $\PP$}.

The $k$-symmetric $(n-p)$-forms $\{\omega_i\}$ patch together to form a global section $\omega=\{\omega_i\}\in H^0(\PP,Sym^{k}\Omega^{n-p}_{\PP}\otimes \mathcal L)$ where $\mathcal L$ is the line bundle over $\PP$ determined by the cocycle $\{g_{ij}\}$. The \textbf{singular set} of $\WW$, denoted by ${\rm{Sing}}(\WW)$, is the zero set of the twisted $k$-symmetric $(n-p)$-form $\omega$. The \textbf{degree} of $\WW$, denoted by $\deg(\WW)$, is geometrically defined as the degree of the tangency locus between $\WW$ and a generic $\mathbb P^{n-p}$ linearly embedded in $\PP$. If $i:\mathbb P^{n-p} \hookrightarrow  \PP$ is the inclusion then the degree of $\WW$ is the degree of the zero divisor of the twisted $k$-symmetric $(n-p)$-form $i^*\omega\in H^0(\mathbb P^{n-p},Sym^{k}\Omega^{n-p}_{\mathbb P^{n-p}}\otimes \mathcal L|_{\mathbb P^{n-p}})$. Since $\Omega^{n-p}_{\mathbb P^{n-p}}=\mathcal O_{\mathbb P^{n-p}}(-n+p-1)$ it follows that $\mathcal L=\mathcal O_{\mathbb P^{n}}(\deg(\WW)+k(n-p)+k)$.

We say that $x\in\PP$ is a \textbf{smooth point} of $\WW$, for short $x\in \WW_{sm}$, if $x\notin {\rm{Sing}}(\WW)$ and the germ of $\omega$ at $x$ satisfies the conditions described in $(3)$ and $(4)$ above. For any smooth point $x$ of $\WW$ we have $k$ distinct (not necessarily in general position) linearly embedded subspaces of dimension $p$ passing through $x$. Each one of these subspaces will be called $p$-plane tangent to $\WW$ at $x$ and denoted by $T^1_x\WW,...,T^k_x\WW$.

To any $k$--distribution $\WW$ of dimension $p$ we can associate the subvariety $S_{\WW}$ of codimension $p$ of $M$ defined as $$S_{\WW}=\overline{\{(x,H)\in M: x \in \WW_{sm} \hspace{0.1cm} and \hspace{0.1cm} \exists \hspace{0.1cm} 1 \leq i \leq k ,\hspace{0.1cm} H\supset T^i_x\WW \hspace{0.1cm}  \}}.$$

The \textbf{characteristic numbers of $\WW$} are by definition the $p+1$ integers $$d_{i}= \int_M[S_{\WW}]\cdot[{\rm{Con}}(\Pp^{n-p-1+i})]\cdot h^{n-p-1} $$
with $i$ ranging from $0$ to $p$. Notice that $d_i$ is the degree of the tangency locus between $\WW$ and a generic $\mathbb P^{n-p+i-1}$. In particular $d_{0}=k$ and $d_{1}$ is the degree of $\WW$, that is $d_{1}=\deg(\WW)$. We remark that in the case $p=n-1$ we arrive in the same definition of \cite[Section 1.4.1]{PiPer}.

\begin{lemma}
The class of $S_{\WW} \subseteq M$ is given by $[S_{\WW}]=d_{p}h^{p}+ \ldots + d_{1}h \hh^{p-1}+ d_{0} \hh^{p}.$  
\end{lemma}
\begin{proof}
It follows from Lemma \ref{conormal-de-Pi} and definition of characteristic numbers. 
\end{proof}

\section{Relations on the characteristic numbers for k-distributions and invariant subvarieties}
Let $\mathcal{D}: \emptyset=L_{n+1} \subseteq L_{n} \subseteq \ldots \subseteq L_{1}\subseteq L_{0}=\PP$ be a flag of linearly embedded subspaces, where $L_i$ has codimension $i$. For each $i\in \{0,...,n+1\}$ we fix the set $\HH_{i}$ of hyperplanes containing $L_{i}$; it  corresponds to a $(i-1)$--dimensional linear space on $\Pd$. Therefore the class of its associated variety $S_{\HH_{i}}=\check{\pi}^{-1}(\HH_{i}) \subseteq M$ is 
$$
[S_{\HH_{i}}]=\hh^{n-i+1}.
$$ 
Now, for a projective subvariety $V \subseteq \PP$ of dimension $q$ and $j \in \{0, \ldots, q\}$ we denote by $P^{V}_{j}= {\rm{tang}}(V,\HH_{q-j+2})$, where $ {\rm{tang}}(V,\HH_{i}):= \pi ({\rm{Con}}(V)\cap S_{\HH_{i}})$. On the other hand $P^{V}_j$ can be seen as pre-image of a Schubert cycle in the Grassmannian by the Gauss map of $V$. To be more precise let $\mathbb G(q,n)$ be the Grassmannian of $q$--dimensional linear spaces of $\mathbb P^n$ and consider the Schubert cycle of codimension $j$
$$
\sigma_j^q=\sigma_j^q(L_{q-j+2})=\{\Gamma\in\mathbb G(q,n)\;:\; \dim(\Gamma\cap L_{q-j+2})\ge j-1\}.
$$
If $\mathcal G_V:V\dashrightarrow \mathbb G(q,n)$ is the natural Gauss map associated to $V$ which sends a smooth point $x\in V_{sm}$ to the tangent space $T_xV$  then $P^{V}_j=\overline{{\mathcal G_V}_{|_{V_{sm}}}^{-1}(\sigma_j^q)}$. These are the polar classes of the variety $V$ defined in \cite{Pi}. It follows from the transversality of a general translate (cf. \cite{Kle}) that for a generic flag, $P^{V}_j$ is equidimensional and its dimension is $q-j$. See \cite{Pi} for details.

In the same spirit, for a $k$--distribution $\WW$ of dimension $p$ and $j \in \{1, \ldots, p+1\}$ we set $P^{\WW}_{j}:= {\rm{tang}}(\WW,\HH_{p-j+2})$ where ${\rm{tang}}(\WW,\HH_{i}):= \pi (S_{\WW}\cap S_{\HH_{i}})$. When $k=1$ we obtain the polar classes of the distribution $\WW$ given in \cite{Mol} and also in \cite{FP}. 

In order to define the Gauss map associated to the distribution we consider $X=\mathbb G(p,n)^k/S_k$ the quotient of $\mathbb G(p,n)^k=\mathbb G(p,n)\times\cdots\times\mathbb G(p,n)$ by the equivalence relation which identifies two points $(\Lambda_1,...,\Lambda_k)$ and $(\Lambda_{\tau(1)},...,\Lambda_{\tau(k)})$ , where $\tau\in S_k$ (the symmetric group in $k$ elements). Then we define the Gauss map 
\begin{eqnarray*}
\mathcal G_{\mathcal W}: \mathbb P^n &\dashrightarrow& X \\
            x   &\mapsto& [T_x^1\mathcal W,...,T_x^k\mathcal W].
\end{eqnarray*}
Since $\mathcal W$ is given locally by $k$ holomorphic distributions of dimension $p$ on the complement of a Zariski closed set, each coordinate of $\mathcal G_{\mathcal W}$ is locally the Gauss map associated to one of these distributions. Therefore $\mathcal G_{\mathcal W}$ is a rational map.

Let us consider the Schubert cycle
$$
\sigma^p_j=\sigma^p_j(L_{p-j+2})=\{\Lambda\in\mathbb G(p,n)\;:\; \dim(\Lambda\cap L_{p-j+2})\ge j-1\}
$$
and the respective closed set in the quotient
$$
\Sigma^p_j=\Sigma^p_j(L_{p-j+2})=\sigma^p_j\times\mathbb G(p,n)^{k-1}/S_{k} \subset X.
$$

If $U$ is the maximal Zariski open set where $\mathcal G_{\WW}$ is regular, it is not hard to see that $P^{\WW}_j=\overline{{\mathcal G_{\WW}}^{-1}_{|_U}(\Sigma^p_j)}$.


\begin{proposition}\label{P:grau}
If $a_{0}, \ldots , a_{n}$ and $d_{0}, \ldots, d_{p}$ are the characteristic numbers of the subvariety $V$ and the $k$--distribution $\WW$ respectively, then for any $j \in \{0, \ldots, q\}$ and any $s \in \{1, \ldots, p\}$ we have
$$
\deg(P^{V}_{j})=a_{n-(q-j)}+a_{n-(q-j)-1}, \hspace{0.5cm} \deg(P^{\WW}_{s})=d_{s}+d_{s-1}. 
$$
In particular $\deg(P^{V}_{0})=\deg(V)$ and $\deg(P^{\WW}_{1})=k+\deg(\WW)$. 
\end{proposition}
\begin{proof}
It follows from the facts 
$$\deg(P^{V}_{j})=\int_M[{\rm{Con}}(V)]\cdot [S_{\HH_{q-j+2}}]\cdot h^{q-j}$$
and 
$$\deg(P^{\WW}_{j})=\int_M[S_{\WW}]\cdot [S_{\HH_{p-j+2}}]\cdot h^{n-j}.$$
\end{proof}

Let us assume now that the flag $\DD$ is sufficiently generic. We state now our main result which relates the characteristic numbers of $V$ and $\WW$ when $V$ is $\WW$--invariant. We say that $V$ is $\WW$--invariant if $V \not\subseteq  {\rm{Sing\WW}}$ and $i^*\omega$ vanishes identically, where $i:V \hookrightarrow  \PP$ is the inclusion and $\omega$ is the twisted $k$-symmetric $(n-p)$--form defining $\WW$.

\begin{theorem}\label{T:theorem}
Suppose that $\WW$ is a $k$--distribution of dimension $p$ on $\PP$ admitting an invariant projective subvariety $V$ of dimension $q \geq p$ and fix $m \in \{1, \ldots p\}$. If $j$ is a number between $0$ and $q-p$ such that $P^{V}_{q-p-j+m} \subseteq P^{\WW}_{m}$ then $P^{V}_{q-p-j} \not\subseteq  P^{\WW}_{m}$ and
$$
\frac{a_{n-(p-m+j)}+a_{n-(p-m+j)-1}}{a_{n-(p+j)}+a_{n-(p+j)-1}} \leq d_{m}+d_{m-1}.
$$
In particular the inequality holds true for $j=0$.
 
\end{theorem}
\begin{proof}
Let $j$ be a number between $0$ and $q-p$. To simplify the notation let us fix $\lambda_1=p+j+2$ and $\lambda_2=p-m+2$. Hence $P^V_{q-p-j}=\overline{{\mathcal G_V}_{|_{V_{sm}}}^{-1}(\sigma_{q-p-j}^{q}(L_{\lambda_1}))}$ and 
$P^{\WW}_m=\overline{{\mathcal G_{\WW}}^{-1}_{|_U}(\Sigma^p_m(L_{\lambda_2}))}$. We will first show that for a generic pair $(L_{\lambda_1},L_{\lambda_2})\in \mathbb G(n-\lambda_1,n)\times \mathbb G(n-\lambda_2,n)$ satisfying $L_{\lambda_1}\subset L_{\lambda_2}$, the dimension of $P^V_{q-p-j}\cap P^{\WW}_m$ is at most $p+j-m$. 

Let $\mathbb F\subset\mathbb G(n-\lambda_1,n)\times \mathbb G(n-\lambda_2,n)$ be the closed set of pairs satisfying $L_{\lambda_1}\subset L_{\lambda_2}$ and consider
$$
\mathcal U=\{(L_{\lambda_1},L_{\lambda_2},\Lambda,\Gamma)\in\mathbb F\times X\times \mathbb G(q,n)\;:\;
\Gamma\in\sigma_{q-p-j}^{q}(L_{\lambda_1})\;,\Lambda\in\Sigma^p_m(L_{\lambda_2})\}.
$$
If $\tilde{V}=V_{sm}\cap U$ then $P^V_{q-p-j}\cap P^{\WW}_m\cap \tilde{V}=p_1(\psi^{-1}(L_{\lambda_1},L_{\lambda_2}))$ where $p_1$ and $\psi$ are the morphisms defined below 
\[
 \xymatrix{
 \tilde{V} \times_{X\times\mathbb G(q,n)}\ar@/^0.5cm/@{->}[rr]^{\psi} \mathcal U \ar[d]^{p_1} \ar[r]  & \mathcal U  \ar[d] \ar[r]  & \mathbb F
  \\
  \tilde{V} \ar[r]^{\mathcal G_{\WW}\times\mathcal G_{V} } & X\times\mathbb G(q,n)}
\]
The unlabeled arrows are the corresponding natural projections. Notice that $X\times\mathbb G(q,n)$ is a $\mathrm{aut}(\mathbb
P^n)$-homogeneous space under the natural action. Since the
vertical arrow $\mathcal U \to X\times\mathbb G(q,n)$ is a
$\mathrm{aut}(\mathbb P^n)$-equivariant morphism the transversality
of the general translate (cf. \cite{Kle}) implies that
\begin{eqnarray*}
\dim\tilde{V} \times_{X\times\mathbb G(q,n)}  \mathcal U &=& \dim \tilde{V} + \dim \mathcal U -
\dim X\times\mathbb G(q,n) \\
&=& q+\dim \mathcal U -k\dim\mathbb G(p,n)-\dim\mathbb G(q,n).
\end{eqnarray*}
Since a fiber of the map $\mathcal U \longrightarrow \mathbb F$ is $\Sigma^p_m\times\sigma^q_{q-p-j}$ one obtains
\begin{eqnarray*}
\dim \mathcal U = k\dim\mathbb G(p,n)-m+\dim\mathbb G(q,n)-(q-p-j)+\dim\mathbb F.
\end{eqnarray*}
The map $\psi$ is dominant because by hypothesis given a generic pair $(L_{\lambda_1},L_{\lambda_2})\in \mathbb F$ we can take  $x\in P^{V}_{q-p-j+m}\cap \tilde{V} \subseteq P^{\WW}_{m}\cap P^{V}_{q-p-j}\cap\tilde{V}$. This fact together with the above equalities we obtain  $\dim\psi^{-1}(L_{\lambda_1},L_{\lambda_2})=p+j-m$ for generic pair in $\mathbb F$. Therefore
\begin{eqnarray*}
\dim P^V_{q-p-j}\cap P^{\WW}_m\cap \tilde{V}\le \dim \psi^{-1}(L_{\lambda_1},L_{\lambda_2})=p+j-m.
\end{eqnarray*}

This shows that $P^{V}_{q-p-j} \not\subseteq  P^{\WW}_{m}$. Furthermore, from the fact $P^{V}_{q-p-j+m}\cap\tilde{V}$ is dense in $P^{V}_{q-p-j+m}$ and $P^{V}_{q-p-j+m}$ has pure dimension $p+j-m$ one obtains that each irreducible component of $P^{V}_{q-p-j+m}$ is an irreducible component of $P^{\WW}_{m}\cap P^{V}_{q-p-j}$. To concludes the theorem we have just to apply Bezout' Theorem and Proposition \ref{P:grau}.

\end{proof}

\begin{corollary}\label{C:corollary}
Let $\WW$ be a $k$--distribution of dimension $p$ on $\PP$ and $V$ a smooth invariant hypersurface of degree $d$. Then for each $m\in\{1,...,p\}$ we obtain
$$
(d-1)^m\le d_m+d_{m-1}.
$$
In particular, 
$$
d\le k+\deg(\WW)+1.
$$
\end{corollary}
\begin{proof}
When $V$ is a smooth hypersurface, it is well known that $\deg(P^{V}_{j})=d(d-1)^j$ (cf. \cite{S,T} or \cite{Soa1} for a modern approach). In addition, follows from Theorem \ref{T:theorem} that for each $m\in\{1,...,p\}$ we have 
$$\deg(P^{V}_{n-1-p+m})\le \deg(P^{V}_{n-1-p})(d_m+d_{m-1}).$$ 
\end{proof}

\begin{remark}\rm
This corollary generalize the bound obtained by M. Soares for one dimensional foliations in \cite{Soa}. We also remark that the bound 
$$d \le deg(\WW)+(n-p)+1$$
is proved in \cite{BrMe} for normal crossing hypersurfaces invariant by a $p$--dimensional foliation. 
\end{remark}
 
\begin{remark}\rm
By the classical formulas of \cite{S,T} for the polar classes of a smooth complete intersection $V$ it is possible to obtains more explicit relations (similar to \cite[Corollary 6.3]{Mol}) between the degree of the homogeneous polynomials defining $V$ and the characteristic numbers of $\WW$.
\end{remark}

\begin{remark}\rm\footnote{The autors are grateful to Jorge Vit\'orio Pereira for having pointed out this remark.}
Unlike the case of foliations, we cannot expect to bound the degree of non smooth invariant subvarieties in terms of the degree of the web, even in the case of nodal curves in dimension two. 
To see this let us take the elliptic curve $E=\mathbb{C}/ \langle 1, \tau \rangle$ and consider the foliation $\FF_{\alpha}$ induced by the $1$-form $\omega=dy-\alpha dx$ on the complex torus $X=E \times E$, where $\alpha \in \mathbb{Q}$. Since $X$ smooth we have an embedding $X \hookrightarrow \Pp^{5}$ and  if we fix a leaf $C_{\alpha}$ of $\FF_{\alpha}$ one may take the restriction of a generic linear projection to $\Pp^{2}$, $\pi_{\alpha}:X \rightarrow \Pp^{2}$, such that the image $D_{\alpha}=\pi(C_{\alpha})$ would be an algebraic curve which has only nodal singularities. Projecting the foliation $\FF_{\alpha}$ we obtain a $d$--web $\WW_{\alpha}$, where $d=deg(X)>1$. Observe that 
$$\deg(\WW_{\alpha})={\rm{tang}}(\WW_{\alpha}, L)={\rm{tang}}(\FF_{\alpha},H)= T^{*}\FF_{\alpha}.H+H^{2}$$
where the last equality follows from \cite[proposition 2, page 23]{Br}, $L$ is a generic line in $\Pp^{2}$ and $H$ is a hyperplane section in $X$; on the other hand the cotangent bundle $T^{*}\FF$ is the same for all these foliations, therefore $\deg(\WW_{\alpha})$ does not depend of $\alpha$. Since varying $\alpha$ we can grow the intersection number between $C_{\alpha}$ and the curve $C:=\{0\} \times E \subseteq X$, and therefore also the intersection between $D_{\alpha}$ and the fixed curve $\pi_{\alpha}(C)$, we deduce that $deg(D_{\alpha})$ increase and cannot be bound by the fixed number $\deg(\WW_{\alpha})$. 
\end{remark}

\bibliographystyle{plain}
\bibliography{nc1}

\begin{thebibliography}{10}

\bibitem{Br}
M.~Brunella.
\newblock {\em Birational geometry of foliations}.
\newblock IMPA, 2000.

\bibitem{BrMe}
M.~Brunella and L.G. Mendes.
\newblock Bounding the degree of solutions to pfaff equations.
\newblock {\em Publicacions Matem{\`a}tiques}, 44(2):593--604, 2000.

\bibitem{FP}
T.~Fassarella and J.V. Pereira.
\newblock On the degree of polar transformations. an approach through
  logarithmic foliations.
\newblock {\em Selecta Mathematica, New Series}, 13(2):239--252, 2007.

\bibitem{GH}
P.~Griffiths and J.~Harris.
\newblock {\em Principles of algebraic geometry}, volume 1994.
\newblock Wiley, 1978.

\bibitem{Kle}
S.L. Kleiman.
\newblock The transversality of a general translate.
\newblock {\em Compositio Math}, 28:287--297, 1974.

\bibitem{Mol}
R.S. Mol.
\newblock Classes polaires associ{\'e}es aux distributions holomorphes de
  sous-espaces tangents.
\newblock {\em Bulletin of the Brazilian Mathematical Society}, 37(1):29--48,
  2006.

\bibitem{PiPer}
J.V. Pereira and L.~Pirio.
\newblock {\em An invitation to web geometry}.
\newblock IMPA, 2009.

\bibitem{Pi}
R.~Piene.
\newblock Polar classes of singular varieties.
\newblock {\em Ann. Scient. {\'E}c. Norm. Sup}, 4:247--276, 1978.

\bibitem{Po}
M.H. Poincar{\'e}.
\newblock Sur l'int{\'e}gration alg{\'e}brique des {\'e}quations
  diff{\'e}rentielles du premier ordre et du premier degr{\'e}.
\newblock {\em Rendiconti del Circolo Matematico di Palermo}, 11(1):193--239,
  1897.

\bibitem{S}
F.~Severi.
\newblock Sulle l'intersezione delle variet\`a algebriche e sopra i loro
  caratteri e singolarit\`a projettive.
\newblock {\em Mem. di Torino}, pages 61--118, 1903.

\bibitem{Soa1}
M.G. Soares.
\newblock Projective varieties invariant by one-dimensional foliations.
\newblock {\em Annals of Mathematics}, 152(2):369--382, 2000.

\bibitem{Soa}
M.G. Soares.
\newblock On the geometry of poincar{\'e}'s problem for one-dimensional
  projective foliations.
\newblock {\em Anais da Academia Brasileira de Ci{\^e}ncias}, 73(4):475--482,
  2001.

\bibitem{T}
J.A. Todd.
\newblock The arithmetical invariants of algebraic loci.
\newblock {\em Proceedings of the London Mathematical Society}, 2(1):190, 1938.

\end{thebibliography}

\end{document}